\documentclass[11pt,letterpaper]{amsart} 
\usepackage{lmodern}
\usepackage{palatino, euler, epic,eepic, amssymb, floatflt, microtype}
\usepackage{verbatim,color}
\usepackage{enumerate}
\usepackage[linktocpage,hidelinks]{hyperref}
\hypersetup{
    colorlinks,
    linkcolor={blue!50!black},
    citecolor={green!50!black},
    urlcolor={red!80!black}
}

\usepackage{tikz-cd}

 \newlength{\baseunit}               
 \newcount{\numlines}                
\newtheorem{theoremx}{Theorem}
\newtheorem{theorem}{Theorem}[section]
\newtheorem{proposition}{Proposition}[section]
 
\newtheorem{conjecture}{Conjecture}[section]

\theoremstyle{definition}
\newtheorem{definition}{Definition} 
\newtheorem{lemma}{Lemma}[section]

\newcommand{\bbP}{\mathbb{P}}
\newcommand{\bbF}{\mathbb{F}}
\newcommand{\bbZ}{\mathbb{Z}}

\newcommand{\cA}{\mathcal{A}}
\newcommand{\cO}{\mathcal{O}}

\newcommand{\seq}{\subseteq}

\def\CC{\mathbb{C}}

\def\PP{\mathbb{P}}

\renewcommand{\H}{\mathcal{H}}

\newcommand{\cR}{\mathcal{R}}
\newcommand{\cS}{{\mathscr{S}}}

\newcommand{\propernormal}{%
  \mathrel{\ooalign{$\lneq$\cr\raise.22ex\hbox{$\lhd$}\cr}}}

\newcommand{\lremind}[1]{{\bf[label:  #1]}}
\newcommand{\notation}[1]{}
\renewcommand{\lremind}[1]{{}}

\newcommand{\cut}[1]{}

\newcommand{\rightarrowdbl}{\rightarrow\mathrel{\mkern-14mu}\rightarrow}

\begin{document}
\pagestyle{plain}
\title{The Geometric Syzygy Conjecture in Positive Characteristic}
\author{Michael Kemeny and Peter Yi Wei}
\address{Department of Mathematics, University of Wisconsin-Madison, 480 Lincoln Dr WI 53706}
\email{michael.kemeny@gmail.com}

\address{Department of Mathematics, University of Arkansas, 850 W Dickson Street, Fayetteville, AR 72701}
\email{peterwei.math@gmail.com}

\begin{abstract}We show the geometric syzygy conjecture in positive characteristic. Specifically, if $C$ is a general smooth curve of genus $g$ defined over an algebraically closed field $\mathbb{F}$ of characteristic $p$, then all linear syzygy spaces $\mathrm{K}_{i,1}(C, \omega_C)$, for any integer $i$ are spanned by syzygies of minimal rank $i+1$ provided $p\geq 2g-4$.
 \end{abstract}
\maketitle

\section{Introduction}

Our aim in this paper is to prove the Geometric Syzygy Conjecture for canonical curves of genus $g \geq 4$ over algebraically closed fields $\mathbb{F}$ of characteristic $p\geq 2g-4$. Suppose that $\phi \, : \,X \subseteq \PP^n_{\mathbb{F}}$, where $\mathbb{F}$ is an algebraically closed field of arbitrary characteristic. A huge amount of work has gone into studying the structure of the homogeneous ideal $I_{X/\PP^n}$ considered as a graded $S:=\mathrm{Sym}\,H^0(\mathcal{O}_{\PP^n}(1))$ module. In order to do this, one commonly assumes that $\phi$ is the embedding defined by a very ample line bundle $L$ on $X$ and furthermore, that $\phi$ is \emph{projectively normal}, i.e.\ that the homogeneous coordinate ring of $X$ can be identified with the global section ring
$$\Gamma_L(X):=\bigoplus_{q \in \mathbb{Z}} H^0(X,L^q),$$
considered as a graded $S$ module, and we have a surjection $S \twoheadrightarrow \Gamma_L(X)$. Assuming this, $I_{X/\PP^n}$ is the kernel of this surjection, and by Hilbert's syzygy theorem, we have a finite, minimal graded free resolution
$$0\to F_k \to \cdots \to F_2 \to F_1 \to I_{X/\PP^n} \to 0, \textrm{ of length }k\leq n.$$
The $F_i$ are all free, graded $S$ modules and so we can write them in terms of $\mathbb{F}$ vector spaces $\mathrm{K}_{i,j}(X,L)$, called the \emph{Koszul spaces} as
$$F_i= \bigoplus_{j \in \mathbb{Z}} \mathrm{K}_{i,j}(X,L) \otimes_{\mathbb{F}} S(-i-j).$$

A huge amount of work has gone into understanding this free resolution. In the case where $X$ is a smooth, canonical curve $C$ of genus $g$, we have in many cases a rather good understanding of the graded modules $F_i$. A conjecture regarding these graded modules was first made by Mark Green in 1984, known as Green's Conjecture \cite{green-koszul}. Whilst this is still in general open, in landmark papers by Voisin \cite{V1}, \cite{V2} it has been proven under the assumption that the canonical curve $C$ is a general smooth curve and $\mathbb{F}=\CC$. In characteristic $p$ (for a very near optimal bound), this result was proven in \cite{AFPRW}, and the bound subsequently improved \cite{rs}. The upshot is that, at in characteristic zero or when $\mathrm{char}(\mathbb{F}) \geq \frac{g-1}{2}$ then, assuming $C$ is general, we have a complete understanding of the graded modules $F_i$, \cite{rs}. However, we do not have a complete understanding of the \emph{maps} in the minimal free resolution of $I_{C/ \PP^{g-1}}$, as opposed to whether or not they vanish.

In this direction, Schreyer formulated the \emph{minimal syzygy conjecture}, which tells us that the modules $F_i$ are spanned by elements which arise as restrictions of scrolls in $\PP^{g-1}$ containing the curve (see \cite[chapter 3]{aprodu-nagel}). We can describe the minimal free resolutions of these scrolls giving rise to these elements explicitly using the Eagon--Northcott complex \cite{schreyer}. To make this precise, we first need the notion of \emph{rank} of a syzygy, which can be given in terms of the maps in the minimal free resolution, \cite[\S 2]{bothmer-JPAA}.

We give the simple definition of rank from \cite[chapter 3]{aprodu-nagel}. Assume a polarized variety $(X,L)$ is normally generated. Then we define the \emph{rank} of the syzygy $\alpha\neq 0 \in \mathrm{K}_{i,1}(X,L)$ to be the dimension of the smallest vector space $V\subseteq H^0(X,L)$ such that $\alpha \neq 0 \in \mathrm{K}_{i,1}(X,L,V)\subseteq \mathrm{K}_{i,1}(X,L)$. See \cite[\S 1]{kemeny-geometric-syzygy} for a discussion on these definitions. The rank of 
$\alpha \in \mathrm{K}_{i,1}(X,L)$ is always at least $i+1$. Those syzygies of the minimal rank are particularly simple. They all arise from a rational normal scroll, \cite[\S 3.2]{aprodu-nagel}. Moreover, syzygies of minimal rank can be explicitly constructed \cite[Prop.\ 1.9]{kemeny-geometric-syzygy}. To be explicit, if $\alpha \neq 0$ is a syzygy of minimal rank, there exists a line bundle $M$ on $X$ with $h^0(M)=2$ and $h^0(L \otimes M^{\vee})=i+1$ such that $\alpha$ is in the $1$-dimensional space
$$\bigwedge^{i+1}H^0(L \otimes M^{\vee}) \otimes \frac{H^0(M)}{\bbF\langle s \rangle} \subseteq \mathrm{K}_{i,1}(X,L)$$
where $s \neq 0 \in H^0(M)$, and the inclusion to $\mathrm{K}_{i,1}(X,L)$ is induced by multiplication by $s$.
Related to this, Schreyer has conjectured
\begin{conjecture}[Geometric Syzygy Conjecture]
    Let $C$ be a general curve of genus $g$. Then all linear syzygy spaces $\mathrm{K}_{i,1}(C,\omega_C)$ are spanned by syzygies of minimal rank $i+1$.
\end{conjecture}
    Over $\mathbb{C}$, this conjecture was proven for $g \leq 8$ \cite{bothmer-Transactions} and for general genus \cite{kemeny-geometric-syzygy}. In positive characteristic $p>g$, the conjecture was proven in the important special case $g=2k$ and the last syzygy space $\mathrm{K}_{k-1,1}(C, \omega_C)$ (\cite{yi-wei}). The aim of this paper is to prove it in full over arbitrary fields $\mathbb{F}$ of characteristic at least $2g-4$.
    \begin{theorem}\label{thm-main}
        The Geometric Syzygy Conjecture holds for general curves of genus $g \geq 4$ defined over an algebraically closed field $\mathbb{F}$ of characteristic $p\geq 2g-4$.
    \end{theorem}

    As in \cite{kemeny-geometric-syzygy}, the technique we use is to reduce to the case $g=2k$ and the last linear syzygy space. However, the major difficulty in imitating the method of \cite{kemeny-geometric-syzygy} is that we no longer can ensure that a general polarized K3 surface contains \emph{nodal}, rational curves in the linear system of a generator of $\mathrm{Pic}(X)$. We overcome this obstacle by observing that the space of integral rational curves of a fixed degree is irreducible, which ensures the existence of suitable nodal rational curves that need not lie on a K3 surface.

    In this paper, $\mathbb{F}$ denotes an algebraically closed field of characteristic $p$.\\

\textbf{Acknowledgements} MK was supported by NSF grant DMS-2100782.

\section{Preliminaries}
\subsection{Koszul cohomology}
Let $V$ be a finite-dimensional vector space over a field $\mathbb{F}$, and let $S_V:=\bigoplus_{d\geq 0}\mathrm{Sym}^d(V)$ denote the symmetric algebra of $V$. Let $M=\oplus_{j\geq 0}M_j$ be a graded $S_V$-module. 
\begin{definition}
    The Koszul cohomology $\mathrm{K}_{i,j}(M,V)$ is defined as the middle cohomology of the complex:
    $$\bigwedge^{i+1}V\otimes M_{j-1}\to\bigwedge^i V\otimes M_j \to\bigwedge^{i-1}V\otimes M_{j+1}.$$
\end{definition}

Koszul cohomology, introduced by Green, is a fundamental tool in the study of projective geometry. Let $X$ be a projective variety, $L$ a line bundle and $B$ a coherent sheaf on $X$. We define the graded $S_L:=S_{H^0(X,L)}$-module
$$\Gamma_X(B,L):=\bigoplus_{j\in\bbZ}H^0(X,L^{\otimes j}\otimes B).$$

We then define $\mathrm{K}_{i,j}(X,B,L):=\mathrm{K}_{i,j}(\Gamma_X(B,L),H^0(X,L))$. 

Given a subspace $V\subseteq H^0(X,L)$, we may regard $\Gamma_X(\cO_X,L)$ as a $S_W$-module and define $$\mathrm{K}_{i,j}(X,B,L;V):=\mathrm{K}_{i,j}(\Gamma_X(B,L),V).$$

When $B=\cO_X$, we write $\mathrm{K}_{i,j}(X,L):=\mathrm{K}_{i,j}(X,B,L)$, and similarly $\mathrm{K}_{i,j}(X,L;V)$.

\subsection{Kernel bundle description}
Let $L$ be a very ample line bundle on a variety $X$. Let $V\subseteq H^0(X,L)$ be a base-point free subspace of dimension $r$ that embeds $X\hookrightarrow\bbP^{r-1}$. The case when $V=H^0(X,L)$ corresponds to the complete linear series. The associated \textit{kernel bundle} $M_V$ is a vector bundle of rank $r-1$ defined by the following exact sequence
    $$0\to M_V\to V\otimes\cO_X\to L\to 0.$$ The we have the following three lemmas.
\begin{lemma}
    Suppose $L$ is a globally generated line bundle on $X$. $B$ is a coherent sheaf on $X$. Then
    $$\mathrm{K}_{i,j}(X,B,L)\simeq \mathrm{coker}\big(\wedge^{i+1}H^0(L)\otimes H^0(X,B\otimes L^{j-1})\to H^0(X,B\otimes \wedge^i M_L\otimes L^j)\big)$$
\end{lemma}
\begin{lemma}[\cite{ein-lazarsfeld-asymptotic}]\label{lemma-kernel-bdl} Assume that $H^i(X,B\otimes\cO_X(mL))=0$ for all $i>0$ and $m\geq 0$. Then 
$$\mathrm{K}_{i,j}(X,B,L;V)=H^1(X,\wedge^{i+1}M_V\otimes B\otimes L^{\otimes j-1}),\textrm{ for all }j\geq 1.$$ 
In particular, $\mathrm{K}_{i,1}(X,L;V)=H^1(X,\wedge^{i+1}M_V)$ and $\mathrm{K}_{i,2}(X,L;V)=H^1(\wedge^{i+1}M_V\otimes L)$.
\end{lemma}

\begin{lemma}[\cite{kemeny-geometric-syzygy}]\label{lemma-linear-syzygy}
    Let $(X,L)$ be a projectively normal polarized variety. For any subspace $V\subseteq H^0(X,L)$, the natural map $\mathrm{K}_{i,1}(X,L;V)\to \mathrm{K}_{i,1}(X,L)$ is injective for all $i\geq 0$.
\end{lemma}

\subsection{Brill--Noether theory for curves in characteristic $p$}
Let $f \, : \, \mathcal{C} \to S$ be a flat, projective morphism of $\mathbb{F}$-varieties, whose geometric fibres are integral curves of the same arithmetic genus $p \in \mathbb{Z}_{>0}$. Assume further, that we have a section $\sigma \, :\, S \to \mathcal{C}$, landing in the smooth locus of $\mathcal{C}$.  We need to make the assumption to avoid issues with representability in the Zariski topology, cf.\ \cite{kass-survey}. It is vacuous \'{e}tale locally. As in \cite{altman-kleiman}, we have the compactified Jacobian $\overline{J}^d(\mathcal{C}) \to S$ of rank one, torsion free sheaves on the fibres of $f$ of degree $d$, which is a projective scheme over $S$. We have a universal sheaf $\mathcal{A}$ on $\mathcal{C} \times_S \overline{J}^d(\mathcal{C})$. The determinantal subscheme $$ W^r_d(\mathcal{C}):= \{ x \in \overline{J}^d(\mathcal{C}) \, | \, h^0([x]) \geq r+1 \} $$
is then constructed as in \cite[\S 4]{kempf-schubert} or \cite[\S 2.2]{bhosle-param}, where $[x] $  denotes the point parametrised by a torsion free sheaf on a fibre of $f$. In the special case where $f$ is smooth, $\overline{J}^d(\mathcal{C})=\mathrm{Pic}^d(\mathcal{C})$. In the special case $S=\mathrm{Spec}(\mathbb{F})$, and $\mathcal{C} \to S$ is a single integral curve $C$ of even genus $g=2k$. Then $W^1_{k+1}(C)$ is expected to be zero dimensional and smooth, and if so, we may compute its length \cite[Corollary to Theorem 3]{kempf-schubert}. By \cite{kempf-schubert}, in the case where $C$ is a smooth curve, then, if we assume $W^1_{k+1}(C)$ is zero dimensional and smooth, it has length $\frac{1}{k+1} \binom{2k}{k}$. If $C$ is only an integral curve lying on a surface, which may be deformed to a smooth curve (for instance, when $C$ lies on primitively polarized K3 surface, see \cite[Cor.\ 3.5]{huybrechts-k3}), then, if $W^1_{k+1}(C)$ is zero dimensional and smooth, we see by semicontinuty and by deforming to the smooth case that $W^1_{k+1}(C)$ consists of at least $\frac{1}{k+1} \binom{2k}{k}$ points.

Let $\pi \, : \, \mathcal{C} \to \mathcal{T}$ be a flat family of integral, Gorenstein curves of arithmetic genus $g$. We define $f \, : \mathbb{X}^1_{k+1}:= \mathrm{Hom}_{\pi}^{k+1}(\mathcal{C}, \PP^1_{\mathcal{T}})$ to be the Hom scheme, such that $f$ is flat, with fibre over a point $t \in \mathcal{T}$ consisting of all morphisms $h \, : \, \mathcal{C}_t:=\pi^{-1}(t) \to \PP_{\mathbb{F}}^1$ such that $\chi(h^*(\mathcal{O}_{\PP^1}(1))=k+2-g$. This may be constructed via the relative Hilbert scheme, \cite[Theorem 1.10]{kollar}. There is a natural $\mathrm{Aut}(\PP^1)$ action on $\mathbb{X}^1_{k+1}$ via automorphisms of the target. Notice that there is a natural morphism $\mathbb{X}^1_{k+1} \to W^1_{k+1}(\mathcal{C})$ (arising from the universal property for $\overline{J}^{k+1}(\mathcal{C})$), which is affine and $\mathrm{Aut}(\PP^1)$ invariant. So, as in \cite[Remark II.2.7]{kollar}, the GIT quotient $\mathcal{X}^1_{k+1}:=\mathbb{X}^1_{k+1} // \mathrm{Aut}(\PP^1)$ can be defined.

\section{Moduli Spaces of rational curves in $\PP^{g-1}$} \label{label-mod-spaces}
Let $\bbF=\overline{\bbF}$ be an algebraically closed field of characteristic $p$, let $g\geq 4$ be an integer, and consider $\mathrm{Mor}_{2g-2}(\PP^1_{\bbF},\PP_{\bbF}^{g-1}),$ the set of morphisms $f\, :\,  \PP^1 \to \PP^{g-1}$ with $\deg f^*\mathcal{O}_{\PP^{g-1}}(1)=2g-2$, which is naturally an open subscheme (parametrizing the projectivization of base point free two dimensional vector spaces) of $$\PP\big(H^0(\PP^1,\mathcal{O}_{\PP^1}(2g-2)^{\oplus g})\big).$$
Hence $\mathrm{Mor}_{2g-2}(\PP^1_{\bbF},\PP_{\bbF}^{g-1}),$ is an irreducible, quasi-projective variety. Notice that, if we pick $g$ pairs $x_1,y_1, \ldots, x_{g},y_{g}$ of general points on $\PP^1$, we obtain a nodal rational curve $D$ of arithmetic genus $g$ such that the natural map $\PP^1 \to D$ identifying $x_i$ to $y_i$, $1\leq i \leq g$, is the normalization and further the line bundle $\omega_D$ on $D$, of degree $2g-2$, is very ample and produces a closed embedding $D \hookrightarrow \PP^{g-1}$, \cite[Theorem 3.6]{CFHR}, \cite[Theorem F]{Catanese}. In particular, $M:=\mathrm{Mor}_{2g-2}(\PP^1_{\bbF},\PP_{\bbF}^{g-1})$ always contains points arising from the normalization of nodal rational canonical curves.

Let us recall the map $\delta_s$ from \cite[Lemma 2.2]{kemeny-geometric-syzygy}. Let $C$ be an integral curve and $[A]\in W^1_d(C) \setminus W^2_d(C)$, i.e.\ $A$ is a line bundle on $C$ with $h^0(A)=2$. Choose $s\neq 0 \in \mathrm{H}^0(A)$. For any fixed integer $i$ we have a map
$$\delta_s \, : \, \bigwedge^{i+1}\mathrm{H}^0(\omega_C \otimes A^{\vee}) \otimes \frac{\mathrm{H}^0(A)}{\bbF\langle s \rangle} \to \mathrm{K}_{i,1}(C,\omega_C; \mathrm{H}^0(\omega_C \otimes A^{\vee})).$$ A particular special case will be important for us. Let $C$ be an integral curve of arithmetic genus $g=2k$ and assume every torsion-free sheaf in $W^1_{k+1}(C)$ has precisely two sections. Assume further $C$ is Gorenstein with $\omega_C$ very ample and normally generated. Then we have a morphism $$\Delta_{C,k-1} \, : \, W^1_{k+1}(C) \to \PP(\mathrm{K}_{k-1,1}(C, \omega_C)),$$ \cite[Prop.\ 2.7]{kemeny-geometric-syzygy},\footnote{In the statement of the proposition we assume that $C$ can be embedded into a smooth surface and $\mathbb{F}=\mathbb{C}$, but neither assumption is used beyond needing $C$ to be Gorenstein} which induces the pull-back map on global sections
$$\Delta^*_{C,k-1} \; : \; \mathrm{H}^0(\mathcal{O}_{\PP}(1)) \to \mathrm{H}^0(\Delta^*\mathcal{O}_{\PP}(1)),$$
with $\PP:=\PP(\mathrm{K}_{k-1,1}(C,\omega_C))$. The relationship between this map and the geometric syzygy conjecture is that, if it holds that $\Delta^*_{C,k-1}$ is injective, then the image of $\Delta_{C,k-1}$ cannot lie in a hyperplane, and hence must span all of $\PP(\mathrm{K}_{k-1,1}(C, \omega_C))$. The condition that $\Delta^*_{C,k-1}$ being injective is algebro-geometric in nature.

\begin{lemma} \label{BN finite}
Let $C$ be a general projective curve of genus $g=2k \geq 4$ over an algebraically closed field $\mathbb{F}$ of characteristic $p > g=2k$. Then the Brill--Noether locus $\mathrm{W}^1_{k+1}(C)$ is nonempty, finite and $\mathrm{W}^2_{k+1}(C)=\varnothing$.
\end{lemma}

\begin{proof}
A primitively polarized $K3$ surface over the field $\mathbb{F}$ of genus $g$ exists, see e.g. \cite[Ch.\ 2.4]{huybrechts-k3}. A general hyperplane section $C$ has nonempty, finite $\mathrm{W}^1_{k+1}(C)$ and $\mathrm{W}^2_{k+1}(C)=\varnothing$ by \cite[Prop.\ 1, Prop.\ 4]{yi-wei}. By constructing Brill--Noether loci on moving curves, as explaned in \S \ref{label-mod-spaces}, it follows that the Brill--Noether locus $\mathrm{W}^1_{k+1}(C)$ is finite and $\mathrm{W}^2_{k+1}(C)=\varnothing$ for the general curve $C$.
\end{proof}

We recall the following lemma, which is \cite[Proposition 2.9]{kemeny-geometric-syzygy}. The result applies directly in our context, as the underlying theory holds for arbitrary algebraically closed fields \cite{gieseker}.
\begin{lemma} \label{gen-node-BN}
	Let $D$ be a general integral curve of arithmetic genus $g=2k\geq 4$ with precisely $m$ nodes for $m \leq k-1$. Then $D$ has gonality $k+1$. Further, $W^2_{k+1}(D)=\varnothing$, any closed point $[A] \in W^1_{k+1}(D)$ corresponds to a \emph{locally free sheaf} $A$ on $D$, and $W^1_{k+1}(D)$ is zero-dimensional and reduced.
\end{lemma}

\begin{proposition}\label{prop-geometric-nodal}
    Let $\pi \, : \, \mathcal{C} \to S$ be a flat family of equidimensional $1$ dimensional, projective varieties over a smooth, integral variety $S$, each of even arithmetic genus $2k$, and assume $\pi$ admits a section into the smooth locus of $\mathcal{C}$.
    
    Assume further, that for some closed point $0 \in S$, $\mathcal{C}_0:=\pi^{-1}(0)$ is an integral, Gorenstein curve with empty Brill--Noether variety $W^2_{k+1}(\mathcal{C}_0)$ and with $W^1_{k+1}(\mathcal{C}_0)$ finite and smooth of cardinality $\frac{1}{k+1} \binom {2k}{k}$, with $\omega_{\mathcal{C}_0}$ very ample and normally generated, and the map $\Delta^*_{\mathcal{C}_0,k-1}$ as above is injective. Assume further $\mathcal{C}_0$ may be flatly deformed to a smooth curve and $\mathrm{K}_{k,1}(\mathcal{C}_0,\omega_{\mathcal{C}_0})=0$.
    
    Then there is a nonempty open set $U \seq S$ about $0$ such that $\pi^{-1}(U)$ is a flat family of integral, Gorenstein, curves such that the map $\Delta^*_{\mathcal{C}_s,k-1}$ above is injective and such that $\mathrm{K}_{k,1}(\mathcal{C}_s,\omega_{\mathcal{C}_s})=0$ for all closed points $s \in U$. 
\end{proposition}
\begin{proof}
Being integral is an open condition, as is being Goreinstein \cite[Lemma 48.25.11]{stacks}, so there is a nonempty open set $U \seq S$ such that $\pi^{-1}(U)$ is a flat family of integral, Gorenstein curves. By semicontinuity of Koszul cohomology, see e.g.\ \cite[Lemma 2.2]{kemeny-geometric-syzygy} (this generalizes unchanged to $\mathbb{F}$ general), we may further shrink $U$ to a Zariski open neighbourhood of $0$ such that $\mathrm{K}_{k,1}(\mathcal{C}_s,\omega_{\mathcal{C}_s})=0$ for all closed points $s \in U$. Now by \cite[Lemma 4.10]{kemeny-geometric-syzygy}, which works verbatim in the characteristic $p$ context, we have a vector bundle $\mathcal{K}$ on $U$ with $\mathcal{K}_b \simeq \mathrm{K}_{k,1}(\mathcal{C}_s,\omega_{\mathcal{C}_s})$ for all closed points $s \in U$. We now argue as in \cite[Prop 4.11]{kemeny-geometric-syzygy}. We have a projective morphism $q \, :\, \mathcal{W}^1_{k+1}(\mathcal{C}) \to S$. By the assumption $W^1_{k+1}(\mathcal{C}_0)$ finite, which implies $q$ is a finite morphism after replacing $S$ with a Zariski neighbourhood of the origin. To prove $q$ is flat, it suffices to show $h^0(\mathcal{O}_{\mathcal{C}_s})=h^0(\mathcal{O}_{\mathcal{C}_0})=\frac{1}{k+1}\binom{2k}{k}$. By semicontinuity $h^0(\mathcal{O}_{\mathcal{C}_s}) \leq h^0(\mathcal{O}_{\mathcal{C}_0})=\frac{1}{k+1}\binom{2k}{k}$, whereas, by deforming $\mathcal{C}_s$ to a smooth curve and using the computation of the cardinality of $W^1_{k+1}(\mathcal{C}_s)$ for a smooth curve such that the Brill--Noether loci is smooth and zero-dimensional (as in \S $1$), we see $h^0(\mathcal{O}_{\mathcal{C}_s})\geq \frac{1}{k+1}\binom{2k}{k}$. So we deduce $h^0(\mathcal{O}_{{\mathcal{C}_s}})=h^0(\mathcal{O}_{\mathcal{C}_0})$ and $q$ is flat.

By the assumption of the existence of a smooth section, we have a universal torsion free sheaf $\mathcal{A}$ on $\mathcal{C}\times_S \mathcal{W}^1_{k+1}$, and further $p_* \mathcal{A}$ is a rank two vector bundle $V$ on $\mathcal{W}^1_{k+1}$ (after possibly shrinking it around the central fibre), by Grauert's theorem, where $p$ denotes the projection $\mathcal{C}\times_S \mathcal{W}^1_{k+1} \to \mathcal{W}^1_{k+1}$. Define $\mathcal{X}^1_{k+1} \to \mathcal{W}^1_{k+1}$ to be $\mathrm{Proj}(V^{\vee})$, which we can think of as a space of $g^1_{k+1}$s as in \cite[\S 2]{kemeny-geometric-syzygy}, and which is a $\PP^1$ bundle over $\mathcal{W}^1_{k+1}$. By \cite[Lemma 4.10]{kemeny-geometric-syzygy} (which is unchanged in positive characteristic), since we are assuming $\mathrm{K}_{k,1}(\mathcal{C}_0,\omega_{\mathcal{C}_0})=0$ and $\omega_{\mathcal{C}_0}$ globally generated, after shrinking $U$ we have a vector bundle $\mathcal{K}$ on $U$, with $\mathcal{K}_s \simeq \mathrm{K}_{k-1,1}(\mathcal{C}_s,\omega_{\mathcal{C}_s})$ for any point $s \in U$. As in \cite[Prop.\ 4.11]{kemeny-geometric-syzygy} we can construct a morphism $\gamma \, : \, \mathcal{X}^1_{k+1} \to \mathrm{Proj}(\mathcal{K}^{\vee})$ over $U$. For any point $s \in U$, $\gamma$ induces a morphism $\gamma_s \, : \, \mathcal{X}^1_{k+1}(\mathcal{C}_s) \to \PP_s$ for $\PP_s:=\mathrm{Proj}(\mathcal{K}^{\vee}_s)$, which induces the map
$$\Delta_{\mathcal{C}_s,k-1} \, : \, \mathrm{H}^0(\mathcal{O}_{\PP}(1)) \to \mathrm{H}^0(\gamma_s^*\mathcal{O}_{\PP}(1)) ,$$ which, by \cite[Prop.\ 2.7]{kemeny-geometric-syzygy}, satisfies $\Delta^*_{\mathcal{C}_s,k-1}$ above is injective for all points $s \in U$, after shrinking $U$ if necessary. 

\end{proof}

\begin{proposition} \label{existence-rational}
    Let $g=2k \geq 4$, and let $\mathbb{F}$ be an algebraically closed field of characteristic $p> 2k$. Then there exists an integral, Gorenstein, projective, rational curve $C$ of genus $g$ over $\mathbb{F},$ with $\omega_C$ very ample and normally generated, with $W^2_{k+1}(C)=\emptyset$, with $W^1_{k+1}(C)$ finite and smooth of cardinality $\frac{1}{k+1} \binom {2k}{k}$ and such that the map $$\Delta^*_{C,k-1} \; : \; \mathrm{H}^0(\mathcal{O}_{\PP}(1)) \to \mathrm{H}^0(\Delta^*\mathcal{O}_{\PP}(1)),$$
is injective, with $\PP:=\PP(\mathrm{K}_{k-1,1}(C,\omega_C))$.
\end{proposition}
\begin{proof}
Let $(X,L)$ be a \emph{general}, primitively polarized $K3$ surface of genus $g$ over $\mathbb{F}$. The $K3$ surface $X$, being general, is not supersingular, so there is a rational, but not necessarily integral, curve in $|L|$, \cite[Theorem 18]{bht}. We first claim that $|L|$ contains an integral rational curve. For this, we argue as in \cite{yi-wei}. By work of Ogus \cite{ogus} there is an $\mathbb{F}$-versal deformation $\pi \, : \, (\mathcal{X}, \mathcal{L}) \to \mathcal{T}$ of $(X,L)$, where $\mathcal{T}$ is an irreducible variety (whilst Ogus works with formal schemes, one can construct the moduli space of polarized $K3$ surfaces with usual schemes, \cite{huybrechts-k3}), with geometric general fibre $\mathcal{X}_{\overline{\tau}}$ satisfying $\mathrm{Pic}(\mathcal{X}_{\overline{\tau}})=\mathbb{Z}[\mathcal{L}_{\overline{\tau}}]$. As $\mathcal{X}_{\overline{\tau}}$ is a $K3$ surface of Picard rank one, all divisors in $|\mathcal{L}_{\overline{\tau}}|$ are integral, and further, as $\mathcal{L}_{\overline{\tau}}$ is ample, \cite[Prop.\ 17]{bht} gives a rational, integral curve $R_{\overline{\tau}} \seq \mathcal{X}_{\overline{\tau}}$. Then the image, $R_{\tau} \seq \mathcal{X}_{\tau}$, where $\mathcal{X}_{\tau}$ is the (non-geometric) generic fiber of $\mathcal{X} \to \mathcal{T}$, is also integral and its component is geometrically rational. Furthermore, there exists a divisor $\mathcal{C} \seq \mathcal{X}$ in $|\mathcal{L}|$ specializing to $R_{\tau}$, and all fibres of $\mathcal{C}$ are rational because the property of having geometrically rational components is closed, \cite[Prop.\ II.2, Cor.\ II.2.4]{kollar}. Since integrality is an open condition, as is being Goreinstein \cite[Lemma 48.25.11]{stacks}, there exists a nonempty open $U \seq \mathcal{T}$, such that, for any closed point $u \in U$, the polarized surface $(X_u,L_u)$ admits the integral, Gorenstein, rational curve $C_u$.

Consider the above family of integral curves $\pi: \mathcal{C}\to\mathcal{T}$. After replacing $\mathcal{T}$ with an \'{e}tale neighbourhood is necessary, we may assume $\pi$ admits a smooth section. The geometric generic fiber $C_{\overline{\tau}}$ lies on a K3 surface $X_{\overline{\tau}}$ with $\mathrm{Pic}(X_{\overline{\tau}})=\bbZ[C_{\overline{\tau}}]$. Arguing as in \cite[Lemma 4.2]{kemeny-geometric-syzygy}, $W^1_d(C_{\overline{\tau}})=\varnothing$ for $d\leq k$ and each $[A]\in W^1_{k+1}(C_{\overline{\tau}})$ has $h^0(C_{\overline{\tau}},A)=2$ and is base point free. As in \cite[Lemma 3.2]{yi-wei}, 
as $C_{\overline{\tau}}\to C_{\tau}$ is a flat base change, shrinking $\mathcal{T}$ if necessary, we have a coherent sheaf $\cA$ on $\mathcal{C}$, such that $\cA_t\in W^1_{k+1}(\mathcal{C}_t)$ and $h^0(\mathcal{C}_t,\cA_t)=2$ and $\cA_t$ base point free (notice the existence of universal sections requires a smooth section) for every closed point $t \in \mathcal{T}$. Moreover, $W^1_{k+1}(C_t)$ is reduced  and $|W^1_{k+1}(C_t)|=\frac{1}{k+1}\binom{2k}{k}$ provided $p>g=2k$ (cf.\ \cite[Proposition 4.6]{kemeny-geometric-syzygy},\cite[Proposition 4]{yi-wei}).

In \cite{yi-wei}, as in the proof of \cite[Theorem 3]{yi-wei} in $\S 4$ (which implies \cite[Theorem 1]{yi-wei}), after shrinking $\mathcal{T}$, we also have $\mathrm{K}_{k,1}(X_t,\mathcal{L}_t)=0$ for any point $t\in\mathcal{T}$. Notice that this condition immediately implies that the curves $C_t \in |L_t|$ have $\omega_{C_t}$ very ample and normally generated. Let $(\mathcal{C}\subseteq \mathcal{X}\to \mathcal{T})$ be the family of embedded curves as in the last paragraph. Shrinking $\mathcal{T}$ if necessary, we obtain that every fiber $C_t$ is integral, Gorenstein and rational. We have a vector bundle $\mathcal{K}$ over $\mathcal{T}$, with $\mathcal{K}_t \simeq \mathrm{K}_{k-1,1}(X_t,\mathcal{L}_t)\simeq \mathrm{K}_{k-1,1}(C_t,\omega_{C_t})$ for any point $t \in \mathcal{T}$ (where the last isomorphism is by the Lefschetz theorem \cite{green-koszul}) as in \cite[Lemma 4.10]{kemeny-geometric-syzygy}.  As in the proof of Proposition \ref{prop-geometric-nodal}, we have a flat family $g \, : \,\mathcal{X}_{k+1} \to \mathcal{T} $, whose fibres parametrize all $g^1_{k+1}$s on the rational curve $\mathcal{C}\to \mathcal{T}$. As in \cite[Prop.\ 4.11]{kemeny-geometric-syzygy} we can construct a morphism $\gamma \, : \, \mathcal{X}^1_{k+1} \to \mathrm{Proj}(\mathcal{K}^{\vee})$ over $\mathcal{T}$. For any point $t \in \mathcal{T}$, $\gamma$ induces a morphism $\gamma_t \, : \, \mathcal{X}^1_{k+1}(\mathcal{C}_t) \to \PP_t$ for $\PP_t:=\mathrm{Proj}(\mathcal{K}^{\vee}_t)$, which induces the map
$$\Delta^*_{\mathcal{C}_t,k-1} \, : \, \mathrm{H}^0(\mathcal{O}_{\PP_t}(1)) \to \mathrm{H}^0(\Delta_{\mathcal{C}_t,k-1}^*\mathcal{O}_{\PP_t}(1)) .$$
Shrinking $\mathcal{T}$ if necessary, it suffices to show that 
$$\Delta^*_{\mathcal{C}_{\tau},k-1}: H^0(\mathcal{O}_{\PP_{\tau}}(1))\to H^0(\Delta_{\mathcal{C}_{\tau},k-1}^*\mathcal{O}_{\PP_{\tau}}(1))$$ is injective, where $\tau$ is the generic point of $\mathcal{T}$. Therefore, it is enough to show its base change to its geometric closure $k(\overline{\tau})$:  $\Delta^*_{\mathcal{C}_{\overline{\tau},k-1}}:=\Delta^*_{\mathcal{C}_{\tau},k-1}\otimes k(\overline{\tau})$ is injective. Note that $C_{\overline{\tau}}$ is an integral curve lying on a K3 surface $X_{\overline{\tau}}$ with $\mathrm{Pic}(X_{\overline{\tau}})=\bbZ[C_{\overline{\tau}}]$. The desired injection is obtained by a similar argument in \cite[Proposition 4.9]{kemeny-geometric-syzygy}.
\end{proof}

\begin{proposition} \label{prop-nodal-global-injection}
Set $g=2k \geq 4$, and let $\mathbb{F}$ be an algebraically closed field of characteristic $p> 2k$. There exists an integral, rational, \emph{nodal} curve $C$ of even arithmetic genus $g=2k \geq 4$ such that the pull-back map on global sections
$$\Delta^*_{C,k-1} \; : \; H^0(\mathcal{O}_{\PP}(1)) \to H^0(\Delta^*\mathcal{O}_{\PP}(1)),$$
with $\PP:=\PP(\mathrm{K}_{k-1,1}(C,\omega_C))$, is injective.

\end{proposition} 
\begin{proof}
Recall the notion of seminormalization, \cite[\S I.7.2]{kollar}. In particular, if $X$ is a  finite type scheme over a field $\mathbb{F}$, then the seminormalization $X^{\mathrm{sn}}$ comes with a finite (and therefore proper) \emph{homeomorphism} $\psi \, : \, X^{\mathrm{sn}} \to X$ \cite[\S 29.47]{stacks}; Therefore, topologically speaking, it is often harmless to replace a scheme with its seminormalization. We also need the notion of Chow varieties \cite[I.3, I.4]{kollar}, which can be constructed in positive characteristic and parametrize families of algebraic cycles.

Let $C_0$ be an integral, Gorenstein rational curve as in Proposition \ref{existence-rational}, considered as a subscheme of $\PP_{\mathbb{F}}^{g-1}$ via the very ample line bundle $\omega_{C_0}$. Let $H$ be the Hilbert scheme of closed subschemes of $\PP_{\mathbb{F}}^{g-1}$ with the same Hilbert polynomial as that of the subscheme $C_0 \seq \PP_{\mathbb{F}}^{g-1}$. Further let $\psi 
\, : \, H^{\mathrm{sn}} \to H$ denote the seminormalization of $H$. Pulling the universal family back via $\psi$, we have a universal family $\widetilde{\pi} \, : \, \widetilde{\mathcal{C}} \to H^{\mathrm{sn}}$, with $0 \in H^{\mathrm{sn}}$ a closed point and central fibre $\widetilde{\pi}^{-1}[0] \simeq C_0$. Further after taking an \'{e}tale cover $\phi \, : \, \overline{H} \to H^{\mathrm{sn}}$ the pullback $\pi \, : \, \mathcal{C} \to \overline{H}$ admits a smooth section.

By Proposition \ref{prop-geometric-nodal}, we have a nonempty open set $U \seq \overline{H}$, containing a point $0' \in \phi^{-1}(0)$, parametrizing integral, Gorenstein curves $\mathcal{C}_s$, for $s \in U$, such that the map $\Delta^*_{\mathcal{C}_s,k-1}$ defined above is injective and further $\mathrm{K}_{k,1}(\mathcal{C}_s,\omega_{\mathcal{C}_s})=0$. Since an \'{e}tale map is in particular flat, and therefore open, $\phi(U) \seq H^{\mathrm{sn}}$ is open, and let the closed set $Z \seq H^{\mathrm{sn}}$ denote its complement. Note $0 \in \phi(U)$. By \cite[Theorem I.6.3]{kollar}, we have a morphism $$f \, : \, H^{\mathrm{sn}} \to \mathrm{Chow}(\PP^{g-1}),$$ 
where $\mathrm{Chow}(\PP^{g-1})$ is a component of the Chow variety containing the point parametrized by the cycle $[C_0]$. Further, $f$ is injective near the point $0$, since $C_0$ is integral. Now the Hilbert scheme is proper over $\mathbb{F}$, and the seminormalization map is finite and hence proper, so $H^{\mathrm{sn}}$ is proper over $\mathbb{F}$ and thus $f$ is a proper morphism. So $f(Z)$ is a closed subset of $\mathrm{Chow}(\PP^{g-1})$ and $f(0) \notin f(Z)$ since $f$ is injective at $0$. Define the open set $U'_1 := \mathrm{Chow}(\PP^{g-1}) \setminus f(Z)$, which contains $f(0)$.

Recall the irreducible, quasi projective variety $M:=\mathrm{Mor}_{2g-2}(\PP^1_{\bbF},\PP_{\bbF}^{g-1})$ from earlier which contains a point $x \in M$ corresponding to the normalization $\mu \, : \, \PP^1 \to C_0$ of $C_0$ composed with the canonical map $C_0 \to \PP^{g-1}_{\mathbb{F}}$. Let $x' \in M^{\mathrm{sn}}$ be the corresponding point under the homeomorphism $M^{\mathrm{sn}} \to M$. By \cite[Corollary I.6.9]{kollar}, we have a morphism $$g \, : \, M^{\mathrm{sn}} \to \mathrm{Chow}(\PP^{g-1}),$$ which takes $x$ to $f(0)$. Define $U_1:=g^{-1}(U_1')$, which is a nonempty open subset of $M$ (under the homeomorphism $M^{\mathrm{sn}}\to M$), and $U_1$ contains $x$. One may consider those points in $U_1\subseteq M$ as rational curves $C$ in $\PP^{g-1}_{\bbF}$ satisfying $\Delta_{C,k-1}^*$ injective.

Now there exists an open subset $\widetilde{U} \seq H$ consisting of integral, nodal, canonical curves, and it is clearly nonempty since there exist points corresponding to nodal, rational curves in $M$. Consider the closed subset $f(H^{\mathrm{sn}} \setminus \psi(\widetilde{U}))$ and its complement $U_2':=\mathrm{Chow}(\PP^{g-1}) \setminus f(H^{\mathrm{sn}} \setminus \psi(\widetilde{U}))$, which is a nonempty open set because $f$ is injective on $\psi(\widetilde{U})$ (in fact, as $f$ is injective on $\psi(\widetilde{U})$, we have $U_2'=f(\psi^{-1}(\widetilde{U})$). Now define $U_2:= g^{-1}(U_2')$, which is a nonempty open subset of $M$ (under the homeomorphism $M^{\mathrm{sn}}\to M$). One may consider those points in $U_2\subseteq M$ as rational nodal canonical curves $C\subseteq \PP^{g-1}_{\bbF}$.

Hence there exists a morphism $h \, : \, \PP^1_{\mathbb{F}} \to \PP^{g-1}_{\mathbb{F}}$ with $[h] \in U_1 \cap U_2$. Then $h(\PP^1_{\mathbb{F}}) \seq \PP^{g-1}_{\mathbb{F}}$ is an integral, nodal rational curve with $\Delta^*_{C,k-1}$ injective.
\end{proof}

We conclude this section by showing Green's conjecture for a general nodal curve in positive characteristics.
\begin{proposition}\label{prop-green-conjecture}
    Let $C$ be a general, integral $m$-nodal curve over an algebraically closed field $\bbF$ with $\mathrm{char}(\bbF)=p>0$ of arithmetic genus $g\geq 4$ for any $0\leq m \leq g$. Assume $p\geq\frac{g+4}{2}$. Then $C$ satisfies Green's conjecture. Precisely, $\mathrm{K}_{\lfloor \frac{g}{2} \rfloor,1}(C,\omega_C)=0$,
    or equivalently, we have for $i\leq \lfloor \frac{g}{2} \rfloor-1$,
    $$\mathrm{K}_{i,2}(C,\omega_C)=0.$$
\end{proposition}
\begin{proof}
It is enough to show that there exists a rational nodal curve $C$ of arithmetic genus $g$ satisfying Green's conjecture. Indeed, we can deform the $g$ nodal curve $C$ and smooth out any number of nodes to obtain an integral $m$-nodal curve with $0\leq m \leq g$. Exactly as in the proof of Proposition \ref{existence-rational}, there exists an integral, Gorenstein, projective, rational curve $C$ on a general polarized K3 surface $(X,L)$ of degree $2g-2$, where $C\in |L|$. Moreover, by \cite[Theorem 1]{yi-wei} and its proof, we obtain the desired vanishing $\mathrm{K}_{\lfloor \frac{g}{2} \rfloor,1}(C,\omega_C)=\mathrm{K}_{\lfloor \frac{g}{2} \rfloor,1}(X,L)=0$. To go from Gorenstein to nodal, we argue as in the proof of Proposition \ref{prop-nodal-global-injection}, using the irreducible space $M:=\mathrm{Mor}_{2g-2}(\PP^1_{\bbF},\PP_{\bbF}^{g-1})$. Thus we see that there exists a morphism $h:\bbP^1_{\bbF}\to\bbP^{g-1}_{\bbF}$, such that $h(\bbP^1_{\bbF})\subseteq \bbP^1_{\bbF}$ is an integral, $g$-nodal nodal rational curve satisfying the vanishing $K_{\lfloor \frac{g}{2} \rfloor,1}(C,\omega_C)=0$. 
\end{proof}

\section{Projection of Syzygies}
Let $(X,L)$ be a polarized variety over a field $\bbF$. For any linear subspace $V\subseteq H^0(X,L)$ and an element $0\neq x\in V^{\vee}$, consider the following exact sequence
$$0\to W_x \to V \xrightarrow{\mathrm{ev}_x} \bbF\to 0 $$
Taking exterior powers induces an exact sequence for $i\leq p-1$:
$$0\to \bigwedge^i W_x \to \bigwedge^i V \xrightarrow{\iota_x} \bigwedge^{i-1} W_x \to 0.$$
The natural map $\iota_x\otimes \mathrm{id}:\bigwedge^i V\otimes H^0(L)\to \bigwedge^{i-1}W_x \otimes H^0(L)$ induces \textit{Aprodu's projection} map on Koszul cohomology groups:
\begin{equation}
    \mathrm{pr}_x: \mathrm{K}_{i,1}(X,L;V) \to \mathrm{K}_{i-1,1}(X,L;W_x)
\end{equation}

Let $C$ be an integral, nodal curve of arithmetic genus $g\geq 4$, and assume that
the canonical linear system $\omega_C$ is very ample and $C\subseteq\bbP^{g-1}$ is projectively normal. Assume $C$ has precisely $m$ nodes and no other singularities. Choose general points $x,y\in C_{\mathrm{sm}}$ in the
smooth locus of $C$ and let $D$ be the nodal curve of genus $g+1$ obtained by identifying $x,y$. Then $\omega_D$ is very ample. Let $v\in D$ be the node over $x,y$ and let $\mu: C \to D$ be the partial normalization map. Note that $(D,\omega_D)$ is normally generated \cite[p. 16]{kemeny-geometric-syzygy}. Consider the projection 
$$\pi_v:\bbP^g \dashrightarrow \bbP^{g-1}$$
from the node $v$. Then $C=\overline{\pi_v(D)}\subseteq\bbP^{g-1}$. Since $\mu_*\omega_C\simeq\omega_D\otimes I_v$, we have a short exact sequence
\begin{equation}\label{eqn-exact-1}
    0 \to H^0(C,\omega_C) \to H^0(D,\omega_D) \xrightarrow{\mathrm{ev}_v} \bbF \to 0    
\end{equation}
It induces Aprodu's projection on $\mathrm{K}_{i,1}(D,\omega_D)$ for every integer $i$. By \cite[Lem 3.1]{aprodu-higher}, the map factors through 
\begin{equation}\label{eqn-pr}
    \mathrm{pr}: \mathrm{K}_{i,1}(D,\omega_D) \to \mathrm{K}_{i-1,1}(C,\omega_C).
\end{equation}

 Let $Z\subseteq\bbP^g$ be the cone over $C\subseteq\bbP^{g-1}$ with vertex at $v$. Denote by $\nu:\widetilde{Z}\to Z$ the blow-up at $v$. The strict transformation $D'\subseteq\widetilde{Z}$ is isomorphic to $C$ with $\mu\sim \nu|_{D'}$. We have $\widetilde{Z}\simeq \bbP(\cO\oplus \omega_C)$ and $\cO_{\widetilde{Z}}(D')\simeq \H\otimes \iota^*\cO_C(x+y)$, where $\H$ is the pull-back from $\cO_{\bbP^{g}}(1)$ and $\iota:\bbP(\cO\oplus\omega_C) \to C$ is the fibration (\cite[Lem 3.2]{projecting}). Then $H^0(\widetilde{Z},\H)\simeq H^0(D,\omega_D)$. Note $\H|_{D'}\simeq \mu^*\omega_D\simeq\omega_C(x+y)$. Therefore, we have natural identifications 
\begin{equation}\label{eqn-iso-1}
    \mathrm{K}_{i,1}(C,\omega_C(x+y))\simeq \mathrm{K}_{i,1}(D',\cO_{D'}(1)) \simeq \mathrm{K}_{i,1}(D,\omega_D).    
\end{equation}

Consider the short exact sequence defining the ideal sheaf of $D'\subseteq \widetilde{Z}$:
$$ 0 \to \H^{\vee}(-\iota^*(x+y)) \to \cO_{\widetilde{Z}} \to \cO_{D'} \to 0$$
which induces an exact sequence of $\cS$-modules, where $\cS:=\mathrm{Sym}(H^0(\widetilde{Z},\H))$, (see \cite[\S 3]{kemeny-geometric-syzygy})
{\small{
\begin{equation}\label{eqn-exact-2}
    0 \to  \bigwedge^i H^0(\H)\otimes \left(\bigoplus_{j \in \mathbb{Z}} H^0\bigg( \mathcal{H}^{\otimes j-1}(-\iota^*(x+y)) \bigg)\right) \to \bigwedge^i H^0(\H)\otimes \Gamma_{\widetilde{Z}}(\H) \to \bigwedge^i H^0(\H)\otimes \Gamma_D(\omega_D) \to 0
\end{equation}
}}

The hyperplane section induces an isomorphism (\cite[Prop 4.3]{projecting}, \cite[Lem 3.3]{kemeny-geometric-syzygy})
\begin{equation}\label{eqn-iso-2}
    \mathrm{K}_{i,1}(\widetilde{Z},\H)\simeq \mathrm{K}_{i,1}(C,\omega_C), \quad \mathrm{K}_{i,1}(\widetilde{Z},-\iota^*\cO_C(x+y);\H)\simeq \mathrm{K}_{i,1}(C,-x-y;\omega_C)    
\end{equation}
then (\ref{eqn-exact-2}) induces a natural long exact sequence (\cite[Prop 4.3]{projecting})
\begin{equation}\label{eqn-long-exact-1}
    0 \to \mathrm{K}_{i,1}(C,\omega_C) \to \mathrm{K}_{i,1}(D,\omega_D) \xrightarrow{\delta} \mathrm{K}_{i-1,1}(C,-x-y;\omega_C) \to \mathrm{K}_{i-1,2}(C,\omega_C)
\end{equation}
For each $i\leq p-1$, the sequence (\ref{eqn-exact-1}) induces exact sequence as $\cS\simeq \mathrm{Sym}(H^0(D,\omega_D))$-modules:
\begin{equation}\label{eqn-exact-3}
    0 \to \wedge^i H^0(C,\omega_C) \otimes \Gamma_D(\omega_D)\to \wedge^i H^0(D,\omega_D) \otimes \Gamma_D(\omega_D) \to \wedge^{i-1}H^0(C,\omega_C) \otimes \Gamma_D(\omega_D)\to 0
\end{equation}
It induces a natural long exact sequence which commutes with (\ref{eqn-long-exact-1}):
{\footnotesize{
$$\begin{tikzcd}
    0 \arrow[r] & \mathrm{K}_{i,1}(C,\omega_C) \arrow[r] \arrow[d, "\alpha_{i,1}"] & \mathrm{K}_{i,1}(D,\omega_D) \arrow[r, "\delta"] \arrow[d, "\beta_{i,1}", "\simeq"'] & \mathrm{K}_{i-1,1}(C,-x-y;\omega_C) \arrow[r] \arrow[d, "\epsilon_{i-1,1}"] & \mathrm{K}_{i-1,2}(C,\omega_C) \arrow[d, "\alpha_{i-1,2}"] \\
    0 \arrow[r] & \mathrm{K}_{i,1}(D,\omega_D;H^0(\omega_C)) \arrow[r] & \mathrm{K}_{i,1}(D,\omega_D) \arrow[r, "\mathrm{pr}"] & \mathrm{K}_{i-1,1}(D,\omega_D;H^0(\omega_C)) \arrow[r] & \mathrm{K}_{i-1,2}(D,\omega_D;H^0(\omega_C))
\end{tikzcd}$$}}
The commutativity is obvious once those maps are explicitly defined. Indeed, consider the restriction map $H^0(\widetilde{Z},\H^{\otimes j}) \simeq H^0(Z,\cO_Z(j)) \to H^0(D,\cO_D(j))\simeq H^0(D,\omega_D^{\otimes j})$, the hyperplane map $H^0(\widetilde{Z},\H)\rightarrowdbl H^0(C,\omega_C)$ and the diagram of Koszul complexes commutes:
{\small{
\[
\begin{tikzcd}
   \wedge^{i+1} H^0(\widetilde{Z},\H) \otimes H^0(\H^{\otimes j-1}) \arrow[r] \arrow[d] & \wedge^i H^0(\widetilde{Z},\H) \otimes H^0(\H^{\otimes j}) \arrow[r] \arrow[d] & \wedge^{i-1} H^0(\widetilde{Z},\H) \otimes H^0(\H^{\otimes j+1}) \arrow[d] \\
   \wedge^{i+1} H^0(C,\omega_C) \otimes H^0(\omega_D^{\otimes j-1}) \arrow[r] & \wedge^i H^0(C,\omega_C) \otimes H^0(\omega_D^{\otimes j}) \arrow[r] & \wedge^{i-1} H^0(C,\omega_C) \otimes H^0(\omega_D^{\otimes j+1})
\end{tikzcd}
\]
}}
Taking cohomology groups yields 
$$\alpha_{i,j}: \mathrm{K}_{i,j}(C,\omega_C)\simeq \mathrm{K}_{i,j}(\widetilde{Z},\H) \to \mathrm{K}_{i,j}(D,\omega_D;H^0(\omega_C)).$$ 

The isomorphism $\beta_{i,j}$ is induced by the isomorphism $H^0(\widetilde{Z},\H)\simeq H^0(D,\omega_D)$. 

Lastly, consider the multiplication $H^0(\widetilde{Z}, \H^{\otimes j}(-\iota^*(x+y)))\to H^0(\widetilde{Z}, \H^{\otimes j})\simeq H^0(D,\omega_D^{\otimes j})$ induced by pulling-back $\iota^*$ a section of $\cO_C(x+y)$ and the diagram of Koszul complexes:
{\small{
\[
\begin{tikzcd}
    0 \arrow[r] \arrow[d]  & \wedge^{i-1}H^0(\widetilde{Z},\H)\otimes H^0(\H(-\iota^*(x+y))) \arrow[r] \arrow[d] & \wedge^{i-2}H^0(\widetilde{Z},\H) \otimes H^0(\H^{\otimes 2}(-\iota^*(x+y))) \arrow[d] \\
   \wedge^{i}H^0(C,\omega_C) \arrow[r] & \wedge^{i-1}H^0(C,\omega_C)\otimes H^0(\omega_D) \arrow[r] & \wedge^{i-2}H^0(C,\omega_C)\otimes H^0(\omega_D^{\otimes 2})
\end{tikzcd}
\]
}}
Taking cohomology groups yields 
$$\epsilon_{i-1,1}: \mathrm{K}_{i-1,1}(C,-x-y;\omega_C)\simeq \mathrm{K}_{i-1,1}(\widetilde{Z},-\iota^*\cO_C(x+y);\H)\to \mathrm{K}_{i-1,1}(D,\omega_D;H^0(\omega_C)).$$

Let $\gamma_{x,y}: \mathrm{K}_{i-1,1}(C, -x-y; \omega_C) \to \mathrm{K}_{i-1,1}(C,\omega_C)$ be the natural morphism induced by multiplying the same section of $\cO_C(x+y)$ from before. Note the projection map $\mathrm{pr}$ factors through $\mathrm{K}_{i-1,1}(C,\omega_C)$; see (\ref{eqn-pr}). We have the following commutative diagram
\begin{equation}\label{eqn-cd-1}
    \begin{tikzcd}
    \mathrm{K}_{i,1}(D,\omega_D) \arrow[rd, "\mathrm{pr}"] \arrow[r, "\delta"] & \mathrm{K}_{i-1,1}(C,-x-y;\omega_C) \arrow[d, "\gamma_{x,y}"] \arrow[dr, "\epsilon_{i-1,1}"] \\
    & \mathrm{K}_{i-1,1}(C,\omega_C) \arrow[r, hook] & \mathrm{K}_{i-1,1}(D,\omega_D;H^0(\omega_C))
\end{tikzcd}    
\end{equation}
This proves the following important lemma.
\begin{lemma} \label{im-pr-nodal}
Let $C$ be an integral $m$-nodal curve over $\bbF$ with $\mathrm{char}(\bbF)=p>0$ of arithmetic genus $g$ for any $0 \leq m \leq g$ such that $(C,\omega_C)$ is normally generated. Let $D$ be the $m+1$ nodal curve obtained by identifying general points $x, y$ in the smooth locus of $C$. For any $0\leq i \leq p-1$, the image of $\mathrm{pr}: \mathrm{K}_{i,1}(D,\omega_D) \to \mathrm{K}_{i-1,1}(C,\omega_C)$ is contained in 
$$\gamma_{x,y}(\mathrm{K}_{i-1,1}(C,-x-y , \omega_C))\subseteq \mathrm{K}_{i-1,1}(C,\omega_C).$$
\end{lemma}

Another important lemma is the following:
\begin{lemma} \label{twisted-van}
Let $C$ be a general, integral $m$-nodal curve over $\bbF$ with $\mathrm{char}(\bbF)=p>0$ of arithmetic genus $g\geq 4$ for any $0 \leq m \leq g$. Assume $p\geq \frac{g+4}{2}$. Let $x,y \in C$ be general points in the smooth locus of $C$. For $i \leq \lfloor \frac{g}{2} \rfloor-2$, we have $$\mathrm{K}_{i,2}(C,-x-y, \omega_C)=0.$$
Further, if $i \neq g-2$ then $\mathrm{K}_{i,3}(C,-x-y, \omega_C)=0$.
\end{lemma}
\begin{proof}
Similar in (\ref{eqn-long-exact-1}), we have the following segment in a long exact sequence
$$\to \mathrm{K}_{i+1,2}(D,\omega_D) \to \mathrm{K}_{i,2}(C,-x-y, \omega_C) \to \mathrm{K}_{i,3}(D,\omega_D) \to$$
But $ \mathrm{K}_{i,3}(D,\omega_D)=0$ unless $i=g-1$ for the canonical curve $(D,\omega_D)$, whereas $\mathrm{K}_{i+1,2}(D,\omega_D)=0$ for $i \leq \lfloor \frac{g}{2} \rfloor-2$ by Proposition \ref{prop-green-conjecture}. Moreover, from the exact sequence
$$\to \mathrm{K}_{i+1,3}(D,\omega_D) \to \mathrm{K}_{i,3}(C,-x-y, \omega_C) \to \mathrm{K}_{i,4}(D,\omega_D) \to$$
and the vanishings $\mathrm{K}_{i,4}(D,\omega_D)$, valid for all $i$, $ \mathrm{K}_{i+1,3}(D,\omega_D)=0$, valid for $i+1 \neq (g+1)-2$, we see $\mathrm{K}_{i,3}(C,-x-y, \omega_C)=0$.
\end{proof}

The rest of the section addresses an important result on the multiplication map $\gamma_{x,y}$. We have a short exact sequence
\begin{equation} \label{weird-module-T}
0 \to \Gamma_C(-x-y, \omega_C) \to \Gamma_C(\omega_C) \to \mathbb{T}_{x,y}(C) \to 0
\end{equation}
for some graded $S$-module $\mathbb{T}_{x,y}(C)$.

\begin{proposition} \label{spanned-im-gamma}
Let $C$ be a general, integral $m$-nodal curve over $\bbF$ with $\mathrm{char}(\bbF)=p>0$ of arithmetic genus $g\geq 4$ for any $0 \leq m \leq g$. Let $x,y \in C$ be general points in the smooth locus of $C$. Fix $i \leq \lfloor \frac{g-3}{2} \rfloor $. Assume $p\geq \frac{g+4}{2}$.  As the points $x,y \in C$ vary over $C$, the subspaces $\mathrm{Im}(\gamma_{x,y}) \subseteq \mathrm{K}_{i,1}(C, \omega_C)$ span $\mathrm{K}_{i,1}(C, \omega_C)$.
\end{proposition}
\begin{proof}
We follow from \cite[Prop 3.7]{kemeny-geometric-syzygy}. 

\textit{Claim}: $\mathrm{Im}(\gamma_{x,y}) \cup \mathrm{Im}(\gamma_{s,t})$ spans $\mathrm{K}_{i,1}(C, \omega_C)$ if $x,y, s, t \in C$ are general. 

From the long exact sequence of Koszul cohomology associated to the short exact sequence (\ref{weird-module-T}), we have an exact sequence
$$\mathrm{K}_{i,1}(C,-x-y, \omega_C) \xrightarrow{\gamma_{x,y}} \mathrm{K}_{i,1}(C,\omega_C) \xrightarrow{\alpha} \mathrm{K}_{i,1}( \mathbb{T}_{x,y}(C), H^0(\omega_C)) \to 0,$$
since $\mathrm{K}_{i-1,2}(C,-x-y, \omega_C)=0$ by Lemma \ref{twisted-van}. Thus, to prove the claim, it suffices to show
$$\alpha_{|_{ \mathrm{Im}(\gamma_{s,t})}} \; : \;  \mathrm{Im}(\gamma_{s,t}) \to \mathrm{K}_{i,1}( \mathbb{T}_{x,y}(C), H^0(\omega_C))$$
is surjective.

Let $D'$ be the nodal curve of genus $g+1$ obtained by identifying $s$ and $t$, with the partial normalization map $\mu': C \to D'$. We have a short exact sequence
$$0 \to \Gamma_{D'}(-x-y, \omega_{D'}) \to \Gamma_{D'}(\omega_{D'}) \to \mathbb{T}_{x,y}(D') \to 0$$
of $\overline{S}:=\mathrm{Sym} \left(H^0(D',\omega_{D'}) \right)$-modules. By Lemma \ref{twisted-van} we have a surjection
$$\alpha' \; : \; \mathrm{K}_{i+1,1}(D',\omega_{D'}) \twoheadrightarrow \mathrm{K}_{i+1,1}(\mathbb{T}_{x,y}(D'),H^0(\omega_{D'}))$$
since $i \leq \lfloor \frac{g+1}{2}\rfloor-2=\lfloor \frac{g-3}{2} \rfloor$. We have a short exact sequence
$$0 \to H^0(C,\omega_C) \xrightarrow{\mu'_*} H^0(D',\omega_{D}) \to \bbF \to 0,$$
inducing an inclusion $S \hookrightarrow \overline{S}$ of graded ring. By restriction of scalars, we may consider $\mathbb{T}_{x,y}(D')$ as an $S$-module, denoted $(\mathbb{T}_{x,y}(D'))_{S}$. Then we may identify $(\mathbb{T}_{x,y}(D'))_{S}$ with $ \mathbb{T}_{x,y}(C)$. We have a projection map 
$$\mathrm{pr} \, : \, \mathrm{K}_{i+1,1}(\mathbb{T}_{x,y}(D'), H^0(\omega_{D'})) \to \mathrm{K}_{i,1}( \mathbb{T}_{x,y}(C), H^0(\omega_C)),$$
by \cite[\S 2]{aprodu-higher}, which fits into an exact sequence
$$\to \mathrm{K}_{i+1,1}(\mathbb{T}_{x,y}(D'), H^0(\omega_{D'})) \xrightarrow{\mathrm{pr}} \mathrm{K}_{i,1}( \mathbb{T}_{x,y}(C), H^0(\omega_C)) \to \mathrm{K}_{i,2}( \mathbb{T}_{x,y}(C), H^0(\omega_C)) \to \ldots.$$
We have the exact sequence
$$\to \mathrm{K}_{i,2}(C,\omega_C) \to \mathrm{K}_{i,2}(\mathbb{T}_{x,y}(C),H^0(\omega_C)) \to \mathrm{K}_{i-1,3}(C,-x-y, \omega_C) \to \ldots$$
We have $\mathrm{K}_{i,2}(C,\omega_C)=0$ by Proposition \ref{prop-green-conjecture}, whereas $\mathrm{K}_{i-1,3}(C,-x-y, \omega_C)=0$ by Lemma \ref{twisted-van}. Thus $\mathrm{K}_{i,2}(\mathbb{T}_{x,y}(C),H^0(\omega_C))=0$ and we have a \emph{surjective} map
$$\mathrm{pr} \, : \, \mathrm{K}_{i+1,1}(\mathbb{T}_{x,y}(D'), H^0(\omega_{D'})) \twoheadrightarrow \mathrm{K}_{i,1}( \mathbb{T}_{x,y}(C), H^0(\omega_C)).$$
We have a commutative diagram
$$\begin{tikzcd}
\mathrm{K}_{i+1,1}(D',\omega_{D'}) \arrow[r, two heads, "\alpha' "] \arrow[d, two heads, "\mathrm{pr}"] & \mathrm{K}_{i+1,1}(\mathbb{T}_{x,y}(D'), H^0(\omega_{D'})) \arrow[d, two heads, "\mathrm{pr}"]\\
\mathrm{K}_{i,1}(C,\omega_C) \arrow[r, "\alpha"]  & \mathrm{K}_{i,1}( \mathbb{T}_{x,y}(C), H^0(\omega_C)),
\end{tikzcd}$$
by functoriality of projection maps, \cite[\S 2]{aprodu-higher}. Since
$$\mathrm{pr}(\mathrm{K}_{i+1,1}(D',\omega_{D'})) \subseteq \mathrm{Im}(\gamma_{s,t})\subseteq \mathrm{K}_{i,1}(C,\omega_C)$$
by Lemma \ref{im-pr-nodal}, we see that $\alpha_{|_{ \mathrm{Im}(\gamma_{s,t})}}$ is surjective, as required.
\end{proof}

\section{The Geometric Syzygy Conjecture}
\begin{proposition}\label{prop-lin-combo-min-rk}
    Let $D$ be the nodal curve of arithmetic genus $g(D)=g+1$, obtained by identifying two general points $x,y\in C_{\mathrm{sm}}$ on an integral, nodal curve $C$ of genus $g$ with $\omega_C$ very ample and $(C,\omega_C)$ projectively normal. Let 
    $$0\neq \alpha\in \mathrm{K}_{i,1}(D,\omega_D)$$
    be a syzygy of minimal rank $i+1$. Assume $D$ lies in the smooth locus of the syzygy scheme $X_{\alpha}$ and that $h^0(C,\mu^*L_{\alpha})=h^0(D,L_{\alpha})=2$ with $\mathrm{deg}(L_{\alpha})=g(D)-i$. 
     
    Then $\gamma_{x,y}(\delta(\alpha))$ is a linear combination of syzygies $\sigma$ of minimal rank $i$. Moreover, $C$ lies in the smooth locus of the associated syzygy scheme $X_{\sigma}$ and has associated line bundles $L_{\sigma}\simeq \mu^* L_{\alpha}$. 
\end{proposition}
\begin{proof}
    We follow from \cite[\S 3]{kemeny-geometric-syzygy}. Let $Y_{\alpha}\supseteq X_{\alpha}$ be the cone over $\overline{\pi_v(X_{\alpha})}$. Let $\mu_{\alpha}:\widetilde{Y_{\alpha}} \to Y_{\alpha}$ be the desingularization of $Y_{\alpha}$ by scroll $\widetilde{Y_{\alpha}}$ and $X'_{\alpha}$ be the strict transform of $X_{\alpha}$. We have natural isomorphisms $\mathrm{K}_{i,1}(X'_{\alpha},\cO_{X'_{\alpha}}(1))\simeq \mathrm{K}_{i,1}(X_{\alpha},\cO_{X_{\alpha}}(1))$. Moreover, as $X'_{\alpha}$ being a divisor of $\widetilde{Y_{\alpha}}$, we have $\cO_{\widetilde{Y_{\alpha}}}(X'_{\alpha})\simeq \cO_{\widetilde{Y_{\alpha}}}(\cR)\otimes \cO_{\widetilde{Y_{\alpha}}}(1)$, where $\cO_{\widetilde{Y_{\alpha}}}(1)$ is the hyperplane class, and $R$ is the ruling class on the scroll $\widetilde{Y_{\alpha}}$. There exists a long exact sequence
    $$0 \to \mathrm{K}_{i,1}(\widetilde{Y_{\alpha}},\cO_{\widetilde{Y_{\alpha}}}(1)) \to \mathrm{K}_{i,1}(X'_{\alpha}, \cO_{X'_{\alpha}}(1)) \xrightarrow{\Delta} \mathrm{K}_{i-1,1}(\widetilde{Y_{\alpha}},-\cR; \cO_{\widetilde{Y_{\alpha}}}(1)) \to \mathrm{K}_{i-1,2}(\widetilde{Y},\cO_{\widetilde{Y_{\alpha}}}(1))\to \cdots$$
    Let $s\in H^0(\cO_{\widetilde{Y_{\alpha}}}(\cR))$ be a section such that the image of the ruling $Z(s)$ under $\mu_{\alpha}$ passes through $v$. Then multiplication by $s$ induces a morphism:
    $$\gamma_s: \mathrm{K}_{i,1}(\widetilde{Y_{\alpha}},-\cR; \cO_{\widetilde{Y_{\alpha}}}(1)) \to \mathrm{K}_{i,1}(\widetilde{Y_{\alpha}}, \cO_{\widetilde{Y_{\alpha}}}(1))$$
    As $\widetilde{Z}=\mathrm{Bl}_v\,Z$ naturally lies in $\mathrm{Bl}_v\,Y_{\alpha}$ with a natural birational morphism $g:\mathrm{Bl}_v\,Y_{\alpha} \to \widetilde{Y_{\alpha}}$, we have $g^*(s)|_{D'}\in H^0(C,\mu^*L_{\alpha})$, where $(C\simeq )D'\subseteq \widetilde{Z}$, is the strict transform of $D$. 
    
    Moreover, using isomorphisms (\ref{eqn-iso-1}) and (\ref{eqn-iso-2}),  we have a commutative diagram
    \[\begin{tikzcd}
    \mathrm{K}_{i,1}(X_{\alpha}',\cO_{X'_{\alpha}}(1)) \arrow[r, "\gamma_s\circ\Delta"] \arrow[d, "r_{D'}"] & \mathrm{K}_{i-1,1}(\widetilde{Y_{\alpha}},\cO_{\widetilde{Y_{\alpha}}}(1))  \arrow[d, "r_{\widetilde{Z}}"]  \\
    \mathrm{K}_{i,1}(D,\omega_D) \arrow[r, "\gamma_{x,y}\circ\delta"]     & \mathrm{K}_{i-1,1}(C,\omega_C)                   
    \end{tikzcd}
    \]
    where $r_{D'}$ (respectively $r_{\widetilde{Z}}$) is the restriction to $D'$ (respectively $\widetilde{Z}$). 

    By definition of the syzygy scheme $\mathrm{Syz}(\alpha)=X_{\alpha}$, there exists $\beta\in \mathrm{K}_{i,1}(X_{\alpha},\cO_{X_{\alpha}}(1))\simeq \mathrm{K}_{i,1}(X'_{\alpha},\cO_{X'_{\alpha}}(1))$ with $r_{D'}(\beta)=\alpha$. Then $\gamma_{x,y}(\delta(\alpha))\in\mathrm{Im}(r_{\widetilde{Z}})$. As the linear syzygy of a scroll is generated by minimal rank, and restriction $r_{\widetilde{Z}}$ does not change the rank, $\gamma_{x,y}(\delta(\alpha))$ is a linear combination of syzygies $\sigma:=r_{\widetilde{Z}}(\sigma')$ of rank $i$. Under the isomorphism
    $$\mathrm{K}_{i-1,1}(\widetilde{Y_{\alpha}},\cO_{\widetilde{Y_{\alpha}}}(1))\simeq \mathrm{K}_{i-1,1}(Y_{\alpha},\cO_{Y_{\alpha}}(1))\simeq \mathrm{K}_{i-1,1}(\overline{\pi_v(X_{\alpha}}),\cO_{\overline{\pi_v(X_{\alpha}}}(1))$$
    $r_{\widetilde{Z}}$ becomes a natural inclusion
    $$\mathrm{K}_{i-1,1}(\overline{\pi_v(X_{\alpha}}),\cO_{\overline{\pi_v(X_{\alpha}}}(1)) \hookrightarrow \mathrm{K}_{i-1,1}(C,\omega_C).$$
    Thus $\overline{\pi_v(X_{\alpha})}\subseteq \mathrm{Syz}(\sigma)$ by definition of the syzygy scheme, which is a rational normal scroll of degree $i$ and codimension $i-1$. Therefore, $\mathrm{Syz}(\sigma)\simeq \overline{\pi_v(X_{\alpha})}$ as desired.
\end{proof}

\begin{proposition} \label{induction-step}
Let $C$ be a general, integral $m$-nodal curve over $\bbF$ with $\mathrm{char}(\bbF)=p>0$ of arithmetic genus $g\geq 4$ for any $0 \leq m \leq g$. Assume $p\geq\frac{g+4}{2}$. Let $x,y \in C_{sm}$ be general and let $D$ be the nodal curve obtained by identifying $x$ and $y$. Let $i \leq \lfloor \frac{g-1}{2} \rfloor$. Suppose $\mathrm{K}_{i,1}(D,\omega_D)$ is spanned by syzygies $\alpha$ satisfying the following properties: 
\begin{enumerate}[\normalfont(i)]
\item $\alpha$ has minimal rank $i+1$.
\item The curve $D$ lies in the smooth locus of the syzygy scheme $X_{\alpha}:=\mathrm{Syz}(\alpha)$.
\item The associated line bundle $L_{\alpha}:=\mathcal{O}_D(R)$ satisfies $\deg(L_{\alpha})=g+1-i$ and $h^0(C,\mu^*L_{\alpha})=2$, where $\mu: C \to D$ is the partial normalization map.
\end{enumerate}
Then $\mathrm{K}_{i-1,1}(C,\omega_C)$ is spanned by syzygies $\sigma$ of rank $i$ such that $X_{\sigma}:=\mathrm{Syz}(\sigma)$ is a scroll with $C$ lying in the smooth locus of $X_{\sigma}$ and with the associated line bundle of the form $L_{\sigma}=\mu^*L_{\alpha}$, where $L_{\alpha}$ is associated to a syzygy $\alpha \in \mathrm{K}_{i,1}(D,\omega_D)$ of rank $i+1$.
\end{proposition}
\begin{proof}
    We follow from \cite[Prop 4.1]{kemeny-geometric-syzygy}. By Proposition \ref{spanned-im-gamma}, $\mathrm{K}_{i-1,1}(C,\omega_C)$ is spanned by the subspaces $\mathrm{Im}(\gamma_{x,y}) \subseteq \mathrm{K}_{i-1,1}(C, \omega_C)$ as $x, y \in C_{sm}$ vary. Thus it suffices to show $\mathrm{Im}(\gamma_{x,y})$ is spanned by syzygies $\sigma$ of rank $i$ and with the required properties. We have $\mathrm{K}_{i-1,2}(C,\omega_C)=0$ by Proposition \ref{prop-green-conjecture}, and thus a surjection
$$\delta \; : \; \mathrm{K}_{i,1}(D,\omega_D) \twoheadrightarrow   \mathrm{K}_{i-1,1}(C,-x-y, \omega_C).$$
Since $\mathrm{K}_{i,1}(D,\omega_D)$ is spanned by syzygies $\alpha$ satisfying the properties (i),(ii),(iii), it suffices to show that, for such a syzygy $\alpha$, the syzygy $\gamma_{x,y}(\delta(\alpha))$ is a linear combination of syzygies of rank $i$ with the required properties. This follows from Proposition \ref{prop-lin-combo-min-rk}.
\end{proof}

\begin{proposition} \label{last-ruling-prop}
Let $D$ be an integral $n$-nodal curve of even arithmetic genus $g=2k$ with $\omega_D$ very ample and $(D,\omega_D)$ normally generated. Let $A \in W^1_{k+1}(D)$ be a base-point free line bundle with $h^0(A)=2$. Let $\alpha \neq 0 \in \mathrm{K}_{k-1,1}(D,\omega_D; H^0(\omega_D \otimes A^{-1}))$ for $s \neq 0 \in H^0(A)$. Then the associated syzygy scheme $X_{\alpha}$ is a scroll and $D$ lies in the smooth locus of $X_{\alpha}$, with associated line bundle
$$A \simeq L_{\alpha}:=\mathcal{O}_D(R),$$
where $R$ is the ruling on $X_{\alpha}$.
\end{proposition}

\begin{theorem}
    Let $C$ be a general curve of genus $\geq 4$ over $\bbF$ with $\mathrm{char}(\bbF)=p>0$. Fix $i\geq 1$, then $\mathrm{K}_{i,1}(C,\omega_C)$ is spanned by syzygies of rank $i+1$, provided $\mathrm{char}(\bbF)=p>2g-2i-2$.
\end{theorem}
\begin{proof}
    This result is trivial unless $i\leq \lfloor (g-2)/2\rfloor$ by Green's conjecture in positive characteristic. Fix such an $i$ and set $m=g-2i-2$. Let $x_k, y_k\in C$ be general points for $1\leq k\leq m$. 

    For any $1 \leq j \leq m$, we let $D_j$ be the $j$-nodal curve obtained by identifying $x_i$ to $y_k$ for $1 \leq i \leq j$. We let $p_i \in D_j$ denote the node over $x_k, y_k$ for $1 \leq i \leq j$. Set $D_0:=C$. We let
$$\mu_j \; : \; D_{j-1} \to D_j$$
be the partial normalization map at the node $p_j$. By Proposition \ref{induction-step} it suffices to show that $\mathrm{K}_{i+m,1}(D_m,\omega_{D_m})$ is spanned by syzygies of minimal rank $i+m+1$ such that $D_m$ lies in the smooth locus of the syzygy scheme $X_{\alpha}:=\mathrm{Syz}(\alpha)$ and such that $\deg(L_{\alpha})=g-i$, $h^0(C,\mu^*L_{\alpha})=2$, where $\mu: C \to D_m$ is the normalization map (note that this implies that all pullbacks of $L_{\alpha}$ along partial-normalizations at any number of nodes have two sections). 

Set $\ell=g-i-1$, then $D_m$ has arithmetic genus $2\ell=2g-2i-2$. By Proposition \ref{prop-nodal-global-injection}, 
$$S^* \; : \; H^0(\mathcal{O}_{\PP}(1)) \to H^0(S^*\mathcal{O}_{\PP}(1)),$$
with $\PP:=\PP(\mathrm{K}_{\ell-1,1}(D_{m},\omega_{D_m}))$, is injective since $m=g-2i-2 \leq \ell-1$, and $\mathrm{char}(\bbF)=p>2g-2i-2=2\ell$. Note that $W^1_{\ell+1}(D_m)$ is zero-dimensional and reduced, and each element $[A]\in W^1_{\ell+1}(D_m)$ is a line bundle with exactly two sections by Proposition  \ref{gen-node-BN}. The morphism $S$ is the map
By Proposition \ref{prop-nodal-global-injection}, the image of $S$ does not lie in any hyperplane and hence the image $S(X^1_{\ell+1}(D_m))$ spans $\PP(\mathrm{K}_{\ell-1,1}(D_m,\omega_{D_m}))$. 

Thus $\mathrm{K}_{\ell-1,1}(D_m,\omega_{D_m})$ is spanned by syzygies $\alpha$ which lie in $\mathrm{K}_{\ell-1,1}(D_m,\omega_{D_m}; H^0(\omega_{D_m} \otimes A^{-1})$ for $s \in H^0(A)$, $A \in W^1_{\ell+1}(D_m)$. By Proposition \ref{last-ruling-prop}, such syzygies $\alpha$ are of minimal rank, $D_m$ lies in the smooth locus of the syzygy scheme $X_{\alpha}$, and the associated ruling is the line bundle $A \in W^1_{\ell+1}(D_m)$.

It remains to show that, for any $A \in W^1_{\ell+1}(D_m)$, we have $h^0(C,\mu^*A)=2$. Assume that $h^0(C,\mu^*A)\geq 3$ for some $A \in W^1_{\ell+1}(D_m)$. Notice that, for each $1 \leq i \leq m$, there exists $s_i \in \mu^*H^0(D_m,A)$ with $$s(x_k)=s(y_k)=0.$$ Consider the space $G^1_{\ell+1}(C)$ of $g^1_{\ell+1}$'s, i.e.\ pairs $V \subseteq H^0(C,L)$, where $L \in W^1_{\ell+1}(C)$ and $V$ is a two dimensional, base-point free subspace of the space of global sections of $L$. As $C$ is general, $G^1_{\ell+1}(C)$ is smooth of dimension $\rho(g,1,\ell+1)=m$, \cite[XXI, Prop.\ 6.8]{ACGH2}. 

If $m=0$ then $D_m=C$ and there is nothing to prove, so assume $m>0$, in which case $G^1_{\ell+1}(C)$ is further irreducible, \cite{fulton-laz-connectedness}. Let $\mathcal{Z} \subseteq G^1_{\ell+1}(C)$ be the closed subscheme of pairs $V \subseteq H^0(C,L)$ with $L \in W^2_{\ell+1}(C)$. Then $\mathcal{Z}$ has codimension at least one in $G^1_{\ell+1}(C)$, \cite[IV, Lemma 3.5]{ACGH2}. On the other hand, if $(x_k,y_k)$ are general for $1 \leq k \leq m$, then the locus $\mathcal{T}$ of pairs $[V \subseteq H^0(C,L)] \in \mathcal{Z}$ satisfying the condition 
$$\text{there exist $s_1, \ldots, s_m \in V$ with $s_j(x_k)=s_j(y_k)=0$ for $1 \leq j \leq m$}$$ has codimension at least $m$. Thus $\mathcal{T}=\varnothing$ as required.

To get the characteristic bound in Theorem \ref{thm-main} for all the linear syzygy space $\mathrm{K}_{i,1}(C,\omega_C)$, we simply put $i=1$, then $2g-2i-2=2g-4$.
\end{proof}

{\parskip=12pt 


\begin{thebibliography}{[EHKSY]}
\bibitem[A]{aprodu-higher} M. Aprodu, {\em{On the vanishing of higher syzygies of curves}}, Mathematische Zeitschrift \textbf{241} (2002), 1-15.
\bibitem[ACGH]{ACGH} E. Arbarello, M. Cornalba, P.A. Griffiths and J. Harris, {\em{Geometry of algebraic curves, Volume I}}, Grundlehren der Mathematischen Wissenschaften \textbf{267}, Springer, Heidelberg (1985).
\bibitem[ACGH]{ACGH2} E. Arbarello, M. Cornalba, P.A. Griffiths and J. Harris, {\em{Geometry of algebraic curves, Volume II}}, Springer \textbf{268}, 2011.


\bibitem[AFPRW]{AFPRW} M. Aprodu, G. Farkas, S. Papadima, C. Raicu and J. Weyman, {\em{Koszul modules and Green's conjecture}},  Inventiones Math.\ \textbf{218} (2019), 657-720.





\bibitem[AK]{altman-kleiman-bertini} A. B. Altman and S. L. Kleiman, {\em{Bertini theorems for hypersurface sections containing a subscheme}}, Communications in Algebra \textbf{7}(1979), 775-790.
\bibitem[AK2]{altman-kleiman} A. B. Altman and S. L. Kleiman, {\em{Compactifying the Picard scheme}}, Advances Math. \textbf{35} (1980), 50-112.



\bibitem[AN]{aprodu-nagel} M. Aprodu and  J. Nagel, {\em{Koszul cohomology and algebraic geometry}}, University Lecture Series \textbf{52}, American Mathematical Society, Providence, RI (2010).





\bibitem[BHT]{bht} F. Bogomolov, B. Hassett and Y. Tschinkel, {\em{Constructing rational curves on $K3$ surfaces}}, Duke M.\ J.\ \textbf{157}, 2011.
\bibitem[BP]{bhosle-param} U. N. Bhosle and A. J. Parameswaran, {\em{Picard bundles and Brill–Noether loci in the compactified Jacobian of a nodal curve}}, Internat. Math. Res. Not. \textbf{15} (2014), 4241-4290.



\bibitem[vB]{bothmer-JPAA} H-C Graf v Bothmer, {\em{Generic syzygy schemes}}, J.\ Pure and Applied Algebra \textbf{208} (2007), 867-876.
\bibitem[vB2]{bothmer-Transactions} H-C Graf v Bothmer, {\em{Scrollar syzygies of general canonical curves with genus $\le8$}}, Transactions AMS \textbf{359} (2007), 465-488.

\bibitem[C]{Catanese} F.\ Catanese, {\em{Pluricanonical—Gorenstein—curves}}, Enumerative geometry and classical algebraic geometry (1982), 51-95.
\bibitem[CFHR]{CFHR} F.\ Catanese, M.\ Franciosi, K.\ Hulek and M.\ Reid, {\em{Embeddings of curves and surfaces}}, Nagoya Math.\ J.\ \textbf{154} (1999), 185-220.











\bibitem[EL2]{ein-lazarsfeld-asymptotic} L. Ein and R. Lazarsfeld, {\em{Asymptotic Syzygies of Algebraic Varieties}}, Inventiones Math.\ \textbf{190} (2012), 603-646.

\bibitem[FuLa]{fulton-laz-connectedness} W. Fulton and R. Lazarsfeld, {\em{On the connectedness of degeneracy loci and special divisors}}, Acta Math.\ \textbf{146} (1981), 271-283.

 


\bibitem[Gi]{gieseker} D. Gieseker, {\em{Stable curves and special divisors: Petri's conjecture}}, Inventiones Math.\ \textbf{66}(1982), 251-275.

\bibitem[G1]{green-koszul} M. Green, {\em{Koszul cohomology and the cohomology of projective varieties}}, J.\  Differential Geo.\ \textbf{19} (1984), 125-171.

 
\bibitem[Hu]{huybrechts-k3} D. Huybrechts, {\em{Lectures on K3 surfaces.}}, Cambridge University Press \textbf{158}, 2016.









\bibitem[Ke1]{projecting} M. Kemeny, {\em{Projecting Syzygies of Curves}}, Algebraic Geometry \textbf{7} (2020), 561-580.
\bibitem[Ke2]{kemeny-geometric-syzygy} M. Kemeny, {\em{The Rank of Syzygies of Canonical Curves}}, J. Reine Ang. Math. (Crelle's journal) \textbf{810} (2024), 97 - 138.
















\bibitem[Ka]{kass-survey}J. Kass, {\em{Lecture notes on compactified Jacobians}}, availiable online at author's personal webpage \url{https://people.ucsc.edu/~jelkass/}, 2008.
\bibitem[Ke]{kempf-schubert} G. Kempf, {\em{Schubert methods with an application to algebraic curves}}, Stichting Mathematisch Centrum (1971).
\bibitem[Ko]{kollar} J. Koll\'{a}r, {\em{Rational curves on algebraic varieties}}, Springer Science and Business Media \textbf{Vol. 32}, (2013).


\bibitem[O]{ogus} A.\ Ogues, {\em{Supersingular $K3$ crystals}}, Ast\'{e}rique \textbf{64} (1979).


\bibitem[RS]{rs} C. Raicu and S. Sam, {\em{Bi-graded Koszul modules, K3 carpets and Green's conjecture}}, Comp.\ Math.\  \textbf{158} (2022),  33-56.
\bibitem[Ry]{rydh} D. Rydh {\em{The canonical embedding of an unramified morphism in an étale morphism}}, Math.\ Zeit.\ \textbf{268} (2011), 707-723.
\bibitem[S2]{sernesi-def} E. Sernesi, {\em{Deformations of algebraic schemes}}, Springer, 2007.
\bibitem[SP]{stacks} The Stacks Project Authors, {\em{Stacks Project}}, \url{http://stacks.math.columbia.edu}, 2019.
\bibitem[S]{schreyer} F.-O. Schreyer, {\em{Syzygies of canonical curves and special linear series}}, Math. Ann. \textbf{275} (1986), 105-137.

\bibitem[V1]{V1} C. Voisin, {\em{Green's generic syzygy conjecture for curves of even genus lying on a K3 surface}}, J.\ European Math. Society \textbf{4} (2002), 363-404.
\bibitem[V2]{V2} C. Voisin, {\em{Green's canonical syzygy conjecture for generic curves of odd genus}},  Comp.\ Math.\ \textbf{141} (2005), 1163-1190.
\bibitem[W]{yi-wei} Y. Wei, {\em{Generic Green's Conjecture and Generic Geometric Syzygy Conjecture in Positive Characteristic}}, arXiv:2109.12187

\end{thebibliography}
\end{document}